\documentclass{article}
\usepackage[utf8]{inputenc}

\title{Sufficient and insufficient conditions for the stochastic convergence of Ces\`{a}ro means}
\author{Aurélien F. Bibaut, Alex Luedtke and Mark J. van der Laan}
\date{\today}

\usepackage{mathtools}
\mathtoolsset{showonlyrefs,showmanualtags}
\usepackage{amssymb,amsmath,amsthm}
\usepackage{color}
\usepackage{algorithm}
\usepackage{dsfont}
\usepackage{algpseudocode}
\usepackage{mathtools} 
\usepackage{fullpage}
\usepackage{natbib}
\usepackage{bm}
\usepackage{enumitem}

\newtheorem{theorem*}{Theorem}
\newtheorem{corollary}{Corollary}

\newtheorem{remark}{Remark}
\newtheorem{question}{Question}
\newtheorem{condition}{Condition}
\newtheorem{example}{Example}

\newtheorem{lemma}{Lemma}
\newtheorem{proposition}{Proposition}

\usepackage{setspace}
\newcommand{\Ind}{\bm{1}}
\newcommand{\sign}{\mathrm{sign}}

\newcommand{\pr}{\mathrm{pr}}

\newcommand{\argmin}{\mathop{\arg\min}}

\begin{document}

\maketitle

\begin{abstract}
We study the stochastic convergence of the Ces\`{a}ro mean of a sequence of random variables. These arise naturally in statistical problems that have a sequential component, where the sequence of random variables is typically derived from a sequence of estimators computed on data. We show that establishing a rate of convergence in probability for a sequence is not sufficient in general to establish a rate in probability for its Ces\`{a}ro mean. We also present several sets of conditions on the sequence of random variables that are sufficient to guarantee a rate of convergence for its Ces\`{a}ro mean. We identify common settings in which these sets of conditions hold.
\end{abstract}


\section{Introduction}


The following fact is well known \citep{cauchy1821, cesaro1888} for deterministic real-valued sequences $(x_n)_{n \geq 1}$:
\begin{align}
n^\beta x_n \rightarrow 0 \textnormal{ for some $\beta\ge 0$}\ \ \implies\ \ n^\beta\bar{x}_n:= n^\beta \frac{1}{n} \sum_{i=1}^n x_n \rightarrow 0. \label{eq:deterministic}
\end{align}
In this note, we investigate the extent to which this kind of result carries over to a sequence $(X_n)_{n\geq 1}$ of random variables defined on a complete probability space $(\mathcal{X}, \mathcal{A}, P)$.  Specifically, we aim to answer the following questions:

\begin{question}
\label{question:insufficient} Is $n^\beta X_n\overset{p}{\rightarrow} 0$ sufficient to ensure that $n^\beta \bar{X}_n:= n^\beta \frac{1}{n}\sum_{i=1}^n X_i\overset{p}{\rightarrow} 0$?
\end{question}

\begin{question}
\label{question:sufficient} Do reasonable conditions on $(X_n)_{n \geq 1}$ imply the convergence of $n^\beta \bar{X}_n$ in probability? almost surely? in mean?
\end{question}

\begin{question}\label{question:exponential} Does knowing that $n^\beta X_n$ satisfies an exponential tail bound imply a similar bound for $n^\beta \bar{X}_n$?
\end{question}

Generalizing the deterministic result \eqref{eq:deterministic} to the stochastic case is important in many statistical problems that have an online or sequential component \citep[e.g.,][]{luedtke2016}. In these settings, $X_n$ is often a function of an estimator computed on data available at time $n$. For example, $X_n$ may be equal to $\widehat{\theta}_n - \theta_0$, where $\theta_0$ is a scalar statistical parameter and $\widehat{\theta}_n$ is an estimator of $\theta_0$ based on the first $n$ observations. Alternatively, $X_n$ may be an excess risk $R(\widehat{\theta}_n) - \inf_{\theta\in \Theta} R(\theta)$, where $R$ a risk function and $\Theta$ is an indexing set. 

Guarantees for estimators are generally stated in terms of some form of stochastic convergence, where the type of convergence established varies depending on the setting. For example, results for empirical risk minimizers (also called minimum contrast estimators or M-estimators) have been given in terms of rates in probability \citep[e.g.,][]{vdV-Wellner-1996} and exponential tail bounds on excess risks \citep[e.g.,][]{bartlett2005,bartlett-jordan-mcauliffe2006}. Convergence rates for kernel density and kernel regression estimators are often given in probability \citep[see e.g.][]{hansen2008}, in mean squared error \citep[see e.g.][]{tsybakov2008}, or almost surely \citep[see e.g.][]{hansen2008}. 

Section~\ref{section:counterexample} answers \ref{question:insufficient} in the negative via a counterexample. Section~\ref{section:l1_conv} answers \ref{question:sufficient} in the affirmative for convergence in mean, and Section~\ref{section:as_conv} similarly answers this question for almost sure convergence. Since convergence in mean or convergence almost surely imply convergence in probability, these sections also yield reasonable conditions for the convergence in probability of $n^\beta \bar{X}_n$. Section~\ref{section:exponential_deviation_bounds} answers \ref{question:exponential} in the affirmative, and also evaluates the implications of this finding for empirical risk minimizers.


Whenever we do not make it explicit in the notation, we use the convention that probabilistic notions are with respect to the measure $P$. This convention is applied to expectations $E$, almost sure convergence, convergence in mean, and $L^r(P)$ norms $\|\cdot\|_r$. Here we recall that $\|f\|_r:=\{\int |f(\omega)|^r dP(\omega)\}^{1/r}$ when $r\in(1,\infty)$ and that $\|f\|_\infty$ denotes the $P$-essential supremum. We call the sequence $(X_n)$ uniformly bounded if $(\|X_n\|_\infty)_{n\ge 1}$ is a bounded sequence.

\section{Motivating examples}

\subsection{Online estimator of the Bayes risk in binary classification.} Suppose $(X_1,Y_1),\ldots,(X_n,Y_n)$ are $n$ i.i.d. copies of a couple of random variables $(X,Y)$, with $X$ a vector of predictors and $Y \in \{-1,1\}$ a binary label. 
Denote $\eta(x):= \pr (Y=1 \mid X=x)$, and let $f_\eta(x) := \sign\{2 \eta(x) - 1\}$ be the Bayes classifier. For any classifier $f$, consider $\ell(f)(x,y) := \Ind[y \neq \sign\{f(x)\}]$ the 0-1 classification loss of $f$, and let $R(f) := E\{\ell(f)(X,Y)\}$, the corresponding classification risk. Say we want to estimate the Bayes risk $R^*:=R(f_\eta)$. Suppose that $(\widehat{f}_i)_{i \geq 1}$ is an $(\mathcal{H}_i)_{i \geq 1}$-adapted sequence of estimators of the Bayes classifiers $f_\eta$, where $\mathcal{H}_i := \sigma\{(X_1,Y_1),\ldots,(X_i,Y_i)\}$ is the filtration induced by the first $i$ observations. Consider the online estimator $\widehat{R}_n := n^{-1} \sum_{i=1}^n \ell(\widehat{f}_{i-1})(X_i,Y_i)$ of $R^*$. It can be checked that the following decomposition holds:
\begin{align}
\widehat{R}_n - R^* = \frac{1}{n}\sum_{i=1}^n \ell(f_{i-1})(X_i,Y_i) - E\left\lbrace \ell(f_{i-1})(X_i,Y_i) \mid \mathcal{H}_{i-1} \right\rbrace + \frac{1}{n} \sum_{i=1}^n \{ R(\widehat{f}_{i-1}) - R^* \}.
\end{align}
The first average can be easily checked to be $O(n^{-1/2})$ with high probability via Azuma-Hoeffding. We would then like to show that the second average is $o(n^{-1/2})$ in some stochastic sense. It is known that, under some well-studied assumptions, the individual terms $R(\widehat{f}_{i-1}) - R^*$ can be shown to converge faster than $i^{-1/2}$ \citep[see, e.g.,][]{audibert2007}. We would like to be able to prove the same for their average.

\subsection{Online estimator of the mean outcome under missingness at random.} 

Consider $(X,Y) \in \mathbb{R}^d \times \{0,1\}$ to be a random couple, with, for instance, $X$ having the interpretation of an individual's demographics and $Y$ representing a person's vote intention. Suppose that $R$ is a third random variable, representing whether a person's outcome is measured.
We observe i.i.d. copies $Z_1:=(X_1,R_1,R_1 Y_1),\ldots,Z_n:=(X_n,R_n,X_n Y_n)$ of $Z:=(X, R, R Y)$.

The objective is to estimate $\Psi(P) := E_P\{E_P(Y \mid R=1,X)\}$, which under some assumptions (missingness at random of the outcome measurement, and non-zero probability of the conditioning event), equals the mean outcome $Y$ across the entire population. Let $(\widehat{Q}_i)_{i\geq 1}$, and $(\widehat{g}_i)_{i \geq 1}$ be sequences of $(\mathcal{H}_i)_{i \geq 1}$-adapted estimators of the conditional missingness probability $g :(r,x) \mapsto \pr_P(R=r \mid X=x)$ and outcome regression function $\bar{Q}:(r,x) \mapsto E_P(Y \mid R=r,X=x)$. Then, denoting $D(P)(x,r,y):= \{g(r,x)\}^{-1} r \{y- \bar{Q}(y,x)\} + Q(1,x) - \Psi(P)$, the online estimator $\widehat{\Psi}_n := n^{-1} \sum_{i=1}^n \Psi(\widehat{P}_{i-1}) + D(\widehat{P}_{i-1})(Z_i)$ admits the following decomposition:
\begin{align}
\widehat{\Psi}_n - \Psi(P) = \frac{1}{n} \sum_{i=1}^n D(\widehat{P}_{i-1})(Z_i) - E_P\{ D(\widehat{P}_{i-1})(Z_i) \mid \mathcal{H}_{i-1} \} + \frac{1}{n} \sum_{i=1}^n \mathrm{Rem}(\widehat{P}_{i-1}, P),
\end{align}
where $\mathrm{Rem}(\widehat{P}_{i-1},P) := E_P[\{g(R,X)\}^{-1} \{\widehat{g}_{i-1}(R,X) -g(R,X)\} \{\widehat{Q}(R,X) -\bar{Q}(R,X)\}]$ is a remainder term. 
If $g$ is uniformly lower bounded over its domain by some $\delta > 0$, then the Cauchy-Schwarz inequality implies that $\mathrm{Rem}(\widehat{P}_{i-1}, P) \leq \delta^{-1} \|\widehat{g}-g\|_2 \|\widehat{Q} - \bar{Q}\|_2$. Convergence guarantees on $\mathrm{Rem}(\widehat{P}_{i-1}, P)$ can therefore be obtained from convergence guarantees on $\widehat{g}$ and $\widehat{Q}$.
Analyzing the online estimator $\widehat{\Psi}_n$ requires 
characterizing the stochastic convergence of the average of the remainder terms.

\section{An example where $n^\beta X_n$ converges to zero in probability, yet $n^\beta \bar{X}_n$ does not}\label{section:counterexample}

The following counterexample shows that $n^\beta X_n\overset{p}{\rightarrow} 0$ does not generally even imply that $(n^\beta \bar{X}_n)_{n\ge 1}$ is uniformly tight, even if the further condition is imposed that the random variables $(X_n)_{n\ge 1}$ are uniformly bounded. Therefore, it is certainly not the case that $n^\beta X_n\overset{p}{\rightarrow} 0$ implies that $n^\beta \bar{X}_n\overset{p}{\rightarrow} 0$.

\begin{proposition}\label{prop:counterexample_conv_proba}
For any $\beta\in(0,1)$ and $b>0$, there exists a sequence of random variables $(X_n)_{n \geq 1}$ such that (1) $n^\beta \bar{X}_n = o_p(1)$ and (2) $|X_n| \leq b$ a.s. for all $n$, and such that $(n^\beta \bar{X}_n)_{n\ge 1}$ is not uniformly tight.
\end{proposition}

\begin{proof}
Without loss of generality, suppose that $b=1$. Fix $\beta\in(0,1)$ and $\alpha\in(0,\beta)$. For all $n \geq 1$, let $p_{n,\alpha} := (2^{\left\lfloor \log_2 n \right\rfloor})^{-\alpha}$. Consider a sequence of independent random variables $(X_n)_{n \geq 1}$ such that, for all $n \geq 1$, $X_n \sim \text{Bernoulli}(p_{n,\alpha})$. The definition of $(p_{n,\alpha})_{n \geq 1}$ ensures that for every $k \geq 1$, $X_{2^{k-1}}, \ldots, X_{2^k-1}$ is a block of $2^{k-1}$ i.i.d. observations with marginal distribution $\text{Bernoulli}(p_{2^{k-1}, \alpha})$.

Observe that, for any $M > 0$, $\pr(X_n \geq M n^{-\beta}) = p_{n,\alpha} \rightarrow 0$, that is, $X_n = o_p(n^{-\beta})$ holds. We will show that $\bar{X}_n$ is not uniformly tight, which implies in particular that it is not true that $\bar{X}_n= O_p(n^{-\beta})$. In what follows, we will denote $\bar{X}_{n_1:n_2} := (n_2 - n_1 + 1)^{-1} \sum_{i=n_1}^{n_2} X_i$.

Fix $M>0$ and $k\ge 1$. For $n=2^k$, we have that
\begin{align}
&\pr\left(\bar{X}_{n-1} \geq  n^{-\beta} M \right)=  \pr \left( \frac{1}{n} \sum_{i=n/2}^{n-1} X_i \geq n^{-\beta} M \right)= \pr \left( \bar{X}_{n/2:n-1} \geq 2 n^{-\beta} M \right) \\
=& \pr \left[ \left\lbrace\frac{n}{p_{n/2, \alpha} ( 1 - p_{n/2, \alpha})}\right\rbrace^{1/2} \left(\bar{X}_{n/2:n-1} - p_{n/2, \alpha} \right) \right. \\
& \qquad \left. \geq   \left\lbrace\frac{n}{p_{n/2, \alpha} ( 1 - p_{n/2, \alpha})}\right\rbrace^{1/2} \left(2 n^{-\beta} M - p_{n/2, \alpha} \right) \right]. \label{eq:survivalLB}
\end{align}
We now use the Berry--Esseen theorem to lower bound the last line in the above display. We have that, for every $i \in \{n/2,\ldots, n-1\}$, $E(X_i) = p_{n/2, \alpha}$, $E\{(X_i - p_{n/2, \alpha})^2\} = p_{n/2,\alpha}(1-p_{n/2, \alpha})$, and
\begin{align}
E\{|X_i -E(X_i)|^3\} &= p_{n/2,\alpha}(1-p_{n/2, \alpha})^3 + (1-p_{n/2, \alpha}) p_{n/2,\alpha}^3\\
&= p_{n/2,\alpha}(1-p_{n/2, \alpha}) \left\lbrace (1 - p_{n/2,\alpha})^2  +p_{n/2,\alpha}^2\right\rbrace \\
&\leq p_{n/2,\alpha} (1-p_{n/2,\alpha}).
\end{align}
Hence, $E\{|X_i -E(X_i)|^3\}/{\rm var}(X_i)\le 1$ for every $i = n/2,\ldots, n-1$. Using the Berry--Esseen bound, and letting $\Phi$ denote the cumulative distribution function of the standard normal distribution, we see that
\begin{align}
& \pr \left[ \left\lbrace\frac{n}{p_{n/2, \alpha} ( 1 - p_{n/2, \alpha})}\right\rbrace^{1/2} \left(\bar{X}_{n/2:n-1} - p_{n/2, \alpha} \right) \right.\\
& \qquad \left. \geq \left\lbrace\frac{n}{p_{n/2, \alpha} ( 1 - p_{n/2, \alpha})}\right\rbrace^{1/2} \left(2n^{-\beta}M - p_{n/2, \alpha} \right) \right] \\
&\geq 1-\Phi \left[\left\lbrace\frac{n}{p_{n/2, \alpha} ( 1 - p_{n/2, \alpha})}\right\rbrace^{1/2} \left(2n^{-\beta}M - p_{n/2, \alpha} \right) \right] - \frac{C}{\sqrt{n}}, \label{eq:BE}
\end{align}
where $C$ is a universal positive constant. Noting that $p_{n/2, \alpha} = (n/2)^{-\alpha}$, we see that, for all $n=2^k$ large enough, $p_{n/2, \alpha}\le 1/2$, and so, for such $n$,
\begin{align*}
\left\lbrace\frac{n}{p_{n/2, \alpha} (1-p_{n/2, \alpha})}\right\rbrace^{1/2}(2 M n^{-\beta} - p_{n/2, \alpha})&\ge 2^{(1-\alpha)/2} n^{(1+\alpha)/2}(2 M n^{-\beta} - [n/2]^{-\alpha}),
\end{align*}
and the right-hand side diverges to $-\infty$ as $n\rightarrow\infty$ since $0<\alpha<\beta<1$. Hence, the right-hand side of \eqref{eq:BE} converges to $1$ as $n\rightarrow\infty$. Combining this with \eqref{eq:survivalLB} and recalling that \eqref{eq:survivalLB} assumed that $n=2^k$ shows that $\pr(\bar{X}_{2^k-1} \geq  2^{-k\beta} M )\rightarrow 1$, and so there exists an infinite subsequence $(n_k)$ of the natural numbers such that $\pr(\bar{X}_{n_k-1} \geq  n_k^{-\beta} M )\rightarrow 1$. As $M>0$ was arbitrary, $n^\beta \bar{X}_n$ is not uniformly tight. 
\end{proof}

\section{Convergence in mean}\label{section:l1_conv}

The following proposition shows that convergence in mean of $n^\beta X_n$ implies convergence in mean of $n^\beta \bar{X}_n$. 

\begin{proposition}\label{prop:L1_conv}
Suppose that $E(|X_n|) = o(n^{-\beta})$. Then $E(|\bar{X}_n|) = o(n^{-\beta})$.
\end{proposition}

\begin{proof}
From the triangle inequality, $n^\beta E(| \bar{X}_n|) \leq n^\beta \times n^{-1} \sum_{i=1}^n E(|X_i|)$. From \eqref{eq:deterministic} applied to the deterministic sequence $\{E(|X_i|)\}_{n \geq 1}$, we have that $n^\beta \times \allowbreak n^{-1} \sum_{i=1}^n  E(|X_i|) \rightarrow 0$, which establishes the claim.
\end{proof}


The above proposition can be restated by recalling that, if $X_n\overset{p}{\rightarrow} 0$, then the convergence in mean of $X_n$ to zero is equivalent to the asymptotic uniform integrability of $(X_n)_{n\ge 1}$ \citep[Theorem~2.20 in][]{van2000asymptotic}. Therefore, the above proposition immediately yields the following corollary.
\begin{corollary}\label{cor:aui}
Suppose that $n^\beta X_n\overset{p}{\rightarrow} 0$ and also that $(n^\beta X_n)_{n\ge 1}$ is asymptotically uniformly integrable, in the sense that
\begin{align}
\lim_{x\rightarrow\infty} \limsup_{n\rightarrow\infty} n^\beta E\{|X_n|\Ind(n^\beta |X_n| > x)\} = 0. \label{eq:aui}
\end{align}
Then, $E(|\bar{X}_n|) = o(n^{-\beta})$.
\end{corollary}
The above can be used to prove the following corollary.
\begin{corollary}\label{eq:auitail}
Fix $r\in (1,\infty]$ and let $q$ denote the H\"{o}lder conjugate of $r$. Suppose that $n^{\beta} X_n\overset{p}{\rightarrow} 0$ in probability and that $r$ is such that $(\|X_n\|_r)_{n\ge 1}$ is a bounded sequence. If
\begin{align}
\lim_{x\rightarrow\infty} \limsup_{n\rightarrow\infty} n^{\beta q} \pr(n^\beta |X_n|>x) = 0, \label{eq:auibdd}
\end{align}
then $E(|\bar{X}_n|) = o(n^{-\beta})$.
\end{corollary}
\begin{proof}
For any $\beta\ge 0$ and $n\ge 1$, H\"{o}lder's inequality shows that $n^\beta E\{|X_n|\Ind(n^\beta |X_n| > x)\} \leq n^\beta \|X_n\|_r \pr(n^\beta |X_n|>x)^{1/q}$. Since $(\|X_n\|_r)_{n\ge 1}$ is bounded and $z\mapsto z^q$ is continuous at zero, \eqref{eq:auibdd} implies \eqref{eq:aui}, and so the result follows by Corollary~\ref{cor:aui}.
\end{proof}
In the special case where $\beta=0$ and $r=\infty$ (and, therefore, $q=1$), \eqref{eq:auibdd} automatically follows from the condition that $X_n\overset{p}{\rightarrow} 0$. Put another way, if $X_n\overset{p}{\rightarrow} 0$ and $(X_n)_{n\ge 1}$ is uniformly bounded, then $E(|\bar{X}_n|) = o(1)$.

Observe that, in the context of the counterexample from the proof of Proposition~\ref{prop:counterexample_conv_proba}, Corollary~\ref{cor:aui} (applied with $r = \infty$) shows that for any $\beta < \alpha$, $\bar{X}_n = O_p(n^{-\beta})$.

\section{Almost sure convergence}\label{section:as_conv}


\begin{proposition}\label{prop:as_conv}
If $n^\beta X_n \rightarrow 0$ almost surely, then $n^\beta \bar{X}_n \rightarrow 0$ almost surely.
\end{proposition}

\begin{proof}
Let $\mathcal{E}$ be the event $\{n^\beta X_n \rightarrow 0 \}$. That $n^\beta X_n \rightarrow 0$ almost surely means that $\pr(\mathcal{E}) = 1$. Suppose that $\mathcal{E}$ holds. Then, by \eqref{eq:deterministic} applied to the realization of the sequence $(X_n)_{n \geq 1}$, we have that $n^{\beta} \bar{X}_n \rightarrow 0$. Therefore, $\pr( n^{\beta} \bar{X}_n \rightarrow 0 ) \geq \pr(\mathcal{E}) \geq 1$, hence the claim.
\end{proof}

\begin{example}[Uniform almost sure convergence of kernel estimators]
Consider $(X_1,Y_1),\ldots,(X_n,Y_n)$ a stationary sequence of observations with $X_i \in \mathbb{R}^d$ and $Y_i \in \mathbb{R}$. For all $x$, let $m(x):= E(Y_1 \mid X_1=x)$ and consider the Nadaraya-Watson estimator $\widehat{m}_n(x):= \sum_{i=1}^n Y_i K\{(X_i-x)/h_n\} / \sum_{i=1}^n K\{(X_i - x) / h_n\}$, where $K:\mathbb{R}^d \to \mathbb{R}$ is a symmetric multivariate kernel and $h_n$ is the bandwidth, which converges to zero. \cite{hansen2008} gives conditions for uniform almost sure convergence of $\widehat{m}_n - m$. In particular, for a compact set $\mathcal{C} \subset \mathbb{R}^d$, and under the conditions of \cite[Theorem 9 in][]{hansen2008}, it holds that $\sup_{x \in \mathcal{C}} | \widehat{m}_n(x)-m(x)| \leq O((\log n / n)^{2/(d+4)})$ almost surely.
\end{example}

We now present two corollaries of Proposition~\ref{prop:as_conv} that provide sufficient conditions for $n^\beta X_n\rightarrow 0$ almost surely, and therefore for $n^\beta \bar{X}_n\rightarrow 0$ almost surely. Like Corollary~\ref{eq:auitail}, the first imposes a bound on the tail of $n^\beta X_n$.

\begin{corollary}
Suppose that, for any $x>0$, there exists $\alpha(x)>0$ such that $\pr(n^\beta X_n > x)=O(n^{-1-\alpha(x)})$. Then, $n^\beta \bar{X}_n \rightarrow 0$ almost surely.
\end{corollary}
\begin{proof}
For $x>0$ and $n\ge 1$, define the event $\mathcal{E}(n,x):=\{n^\beta X_n>x\}$. Because $\pr\{\mathcal{E}(n,x)\}=O(n^{-1-\alpha(x)})$, there exists a constant $C<\infty$ such that $\sum_{n=1}^\infty \pr\{\mathcal{E}(n,x)\} = C \sum_{n=1}^\infty n^{-1-\alpha(x)} < \infty$. Hence, by the Borel-Cantelli lemma, $\pr\{\limsup_n \mathcal{E}(n,x)\}=0$. As $x>0$ was arbitrary, $\pr\{\cap_{k=1}^\infty \limsup_n \mathcal{E}(n,1/k)\}=0$, which implies that $n^{\beta} X_n\rightarrow 0$ almost surely. The result follows by Proposition~\ref{prop:as_conv}.
\end{proof}

The second corollary works in the setting where $(X_n)_{n\ge 1}$ is an adapted process. The corollary imposes a condition that is considerably weaker than the requirement that $n^\beta |X_n|$ almost surely converge to zero, but, in the case where $\beta>0$, is stronger than the condition that $|X_n|$ is a supermartingale. In the case where $\beta=0$, the imposed condition is equivalent to requiring that $|X_n|$ is a supermartingale.


\begin{corollary}
Suppose that $(X_n)_{n\ge 1}$ is a sequence of random variables that is adapted to the filtration $(\mathcal{H}_n)_{n=1}^\infty$, that $n^\beta X_n\overset{p}{\rightarrow} 0$, and that
\begin{align}
(1+1/n)^\beta E(|X_{n+1}| \mid \mathcal{H}_n) \leq |X_n|\ \textnormal{ for all $n\ge 1$.} \label{eq:supermart}
\end{align}
Under these conditions, $n^\beta \bar{X}_n\rightarrow 0$ almost surely.
\end{corollary}

\begin{proof}
Let $Y_n:= n^\beta |X_n|$. Eq.~\ref{eq:supermart} imposes that $(Y_n)_{n=1}^\infty$ is a supermartingale adapted to the filtration $(\mathcal{H}_n)_{n=1}^\infty$. Since $|X_n|$ is nonnegative, $E[Y_n^-]=0<\infty$. Hence, by Doob's martingale convergence theorem, $Y_n$ converges almost surely to a random variable $Y_{\infty}$. Moreover, since $Y_n\overset{p}{\rightarrow} 0$, it must be the case that $Y_{\infty}=0$. Hence, $Y_n\rightarrow 0$ almost surely. Proposition 2 then gives the result.
\end{proof}

Since $1+1/n\leq \exp(1/n)$ for all $n$, the above corollary remains true if \eqref{eq:supermart} is replaced by the condition that $\exp(\beta/n) E(|X_{n+1}|\mid\mathcal{H}_n)\leq |X_n|$ for all $n$.

\section{Exponential deviation bounds}\label{section:exponential_deviation_bounds}

The following result shows that if $X_n$ satisfies an exponential deviation bound, then $\bar{X}_n$ also satisfies such a bound.

\begin{proposition}\label{prop:exp_tail_bound}
Suppose that $(X_n)_{n \geq 1}$ is a sequence of random variables, for which there exists $C_0 \geq 0$, $C_1, C_2 > 0$, $\beta \in (0,1)$, and $\gamma \in (0, 1/\beta)$, such that, for any $x > 0$ and any $n \geq 1$,
\begin{align}
\mathrm{pr} \left(X_n \geq C_0 n^{-\beta} + x \right) \leq C_1 \exp(-C_2 n x^\gamma).
\end{align}
Let $\delta \in (\beta, \min( \gamma^{-1}, 1))$.
Then, there exists a constant $C_4 > 0$ depending only on the constants of the problem ($C_0, C_1, C_2, \beta, \gamma$, and $\delta$), such that, for any $y \geq 1$, it holds that
\begin{align}
\mathrm{pr} \left(\bar{X}_n \geq \frac{C_0}{1 - \beta} n^{-\beta} + \frac{3}{1-\delta} n^{-\delta} y \right) \leq C_4 n^\alpha \exp\left\lbrace -C_2 n^{\alpha ( 1 - \gamma \delta)} y^\gamma \right\rbrace,
\end{align}
with $\alpha := \gamma(1-\delta) / \{\gamma (1 -\delta) + (1 - \gamma \delta)\}$.
\end{proposition}
We defer the proof of the above result to the end of the current section.

Empirical risk mininizers are a common type of estimators for which the excess risk satisfies an exponential tail bound, as the following example shows. This example is a weakened version of Theorem 17 in \cite{bartlett-jordan-mcauliffe2006}.

\begin{example}\label{prop:thm17_barlett2006}
Suppose that $(X_1,Y_1),\ldots,(X_n,Y_n)$ are i.i.d. copies of a a couple of random variables $(X,Y)$ taking values in $\mathcal{X} \times \mathcal{Y}$. Consider a class of functions $\mathcal{F}$ defined as  $\mathcal{F} := B \mathrm{absconv}(\mathcal{G})$, for some constant $B > 0$, and function class $\mathcal{G} \subseteq \{\pm 1\}^{\mathcal{X}}$, where $\mathrm{absconv}$ denotes the absolute convex hull (or symmetric convex hull). Let $\ell$ be a loss on $\mathcal{F}$, that is, a mapping defined on $\mathcal{F}$, such that for all $f \in \mathcal{F}$, $\ell(f)$ is a mapping $\mathcal{X} \times \mathcal{Y} \to \mathbb{R}$. For any $f$, let $R(f):=E\{\ell(f)(X,Y)\}$. Let $\widehat{f}$ be an empirical risk minimizer over $\mathcal{F}$, that is, $f \in \argmin_{f \in \mathcal{F}} \sum_{i=1}^n \ell(f)(X_i,Y_i)$. Let $f^* \in \argmin_{f \in \mathcal{F}} R(f)$, a minimizer of the population risk over $\mathcal{F}$. Suppose that the following conditions are met.
\begin{condition}
There exists $L>0$ such that, for any $x,y \in \mathcal{X} \times \mathcal{Y}$, and any $f_1,f_2 \in \mathcal{F}$, $|\ell(f_1)(x,y) - \ell(f_2)(x,y)| \leq |f_1(x) - f_2(x)|$.
\end{condition}
\begin{condition}
There exists $c>0$ such that, for any $f \in \mathcal{F}$, $E[\{\ell(f)(X,Y) - \ell(f^*)(X,Y)\}^2] \leq c\{R(f) - R(f^*)\}$.
\end{condition}
\begin{condition}
It holds that $d_{VC}(\mathcal{F}) \leq d$, for some $d \geq 1$, where $d_{VC}$ is the Vapnik-Chervonenkis dimension.
\end{condition}
Then, it holds that, for any $x > 0$,
\begin{align}
\pr\left\lbrace R(\widehat{f}) - R(f^*) \geq C_0 n^{-(d+2)/(2d+2)} + x \right\rbrace \leq \exp(-C_1 n x),
\end{align}
for some $C_0, C_1 > 0$ depending on the $B$, $L$, and $c$.
\end{example}

\begin{remark}
Observe that the bound from Example~\ref{prop:thm17_barlett2006} above is of the form $\pr(X_n \geq C_0 n^{-\beta} + x) \leq C_1 \exp(-C_2 n x^\gamma)$ with $\beta \gamma <1$.
\end{remark}

The proof of proposition 4 relies on the following lemma.
\begin{lemma}\label{lemma:unif_exp_bound_from_m}
Suppose that $(X_n)_{n \geq 1}$ is a sequence of random variables, for which there exists $C_0 \geq 0$, $C_1, C_2 > 0$, $\beta \in (0,1)$, and $\gamma \in (0, 1/\beta)$, such that, for any $x > 0$ and any $n \geq 1$,
\begin{align}
\mathrm{pr} \left(X_n \geq C_0 n^{-\beta} + x \right) \leq C_1 \exp(-C_2 n x^\gamma).
\end{align}
Consider $\delta \in (0, \min(\gamma^{-1}, 1))$. There exists a constant $C'_1$ that depends only on the constants of the problem ($C_0, C_1, C_2, \beta, \gamma, \delta)$ such that, for any integer $m \geq 1$, and any real number $y \geq 1$, 
\begin{align}
\pr\left( \exists k \geq m+1: X_k \geq C_0 k^{-\beta} + k^{-\delta} y \right) \leq C'_1 m \exp\left(-C_2 m^{1 - \gamma \delta} y^\gamma \right).
\end{align}
\end{lemma}

\begin{proof}
Let $y \geq 1$. We have that 
\begin{align}
\pr \left( \exists k \geq m+1: X_k \geq C_0 k^{-\beta} + k^{-\delta} y \right) \leq & C_1 \sum_{k \geq m+1} \exp\left( - C_2 k^{1-\gamma \delta} y^\gamma \right) \\
\leq & C_1 \int_m^\infty \exp\left(- C_2 k^{1-\gamma \delta} y^\gamma \right) dk
\end{align}
Making the change of variable $u = k^{1-\gamma \delta} y^\delta$, we obtain
\begin{align}
\pr \left( \exists k \geq m+1: X_k \geq C_0 k^{-\beta} + k^{-\delta} y \right) \leq & C_1 y^{-\gamma/(1-\gamma \delta)} \int_{m^{1-\gamma \delta} y^\gamma}^\infty \exp(-C_2 u) u^{1/(1-\gamma \delta) - 1} du \\
\leq & C_1 y^{-\gamma \kappa} \int_{m^{1-\gamma \delta} y^\gamma}^\infty \exp(-C_2 u) u^{\left\lceil \kappa - 1 \right\rceil} du,
\end{align}
where we denote $\kappa := 1 / (1-\gamma \delta)$. Observe that $\kappa > 1$.

We now prove a general identity for the type of integral that appears in the last line of the above display. Denote, for any integer $q \geq 1$, and real numbers $a \geq 1$, and $c > 0$, $I_q(a, c):= \int_a^\infty \exp(-c u) u^q du$. By integration by parts, we have that
\begin{align}
I_q(a,c) = \exp(-c a) \frac{a^q}{c} + \frac{q}{c} I_{q-1}(a,c).
\end{align}
Reasoning by induction, we obtain that 
\begin{align}
I_q(a,c) =& \exp(-ca) \left( \frac{a^q}{c} + \frac{q a^{q-1}}{c^2} + \ldots \frac{q!}{c^{q+1}} \right) \\
\leq & q \times q! \max\{c^{-1}, c^{-(q+1)}\} a^q \exp(-c a),
\end{align}
where we have used in the last line that $a \geq 1$.

Therefore, denoting $C_3 := C_3(C_2, \kappa) := \left\lceil \kappa - 1 \right\rceil \times \left\lceil \kappa - 1 \right\rceil ! \max\{c^{-1}, c^{-(q+1)}\}$, we have that
\begin{align}
\pr \left( \exists k \geq m+1: X_k \geq C_0 k^{-\beta} + k^{-\delta} y \right) \leq & C_1 C_3 y^{-\gamma \kappa} m^{(1 - \gamma \kappa)\left\lceil \kappa - 1 \right\rceil} y^{\gamma \left\lceil \kappa - 1 \right\rceil} \exp\left(-C_2 m^{1 - \gamma \delta} y^\gamma \right) \\
\leq & C_1 C_3 m \exp\left(-C_2 m^{1-\gamma \delta} y^\delta \right),
\end{align}
where we have used in the last line that $m \geq 1$ and $y \geq 1$.
\end{proof}

We now prove proposition \ref{prop:exp_tail_bound}.

\begin{proof}[of proposition \ref{prop:exp_tail_bound}]
Let $y_1 \geq 1$ and $y \geq 1$, and let $1 \leq m < n$. From lemma \ref{lemma:unif_exp_bound_from_m}, we have that, with probability at least $1- C_1 C_3 \exp(-C_2 y_1^\gamma)$, 
\begin{align}
\frac{1}{n} \sum_{k=2}^m X_k - C_0 k^{-\beta} \leq \frac{1}{1-\delta} \frac{m^{1- \delta}}{n} y_1, \label{eq:pf_Cesaro_exp_first_part_sum}
\end{align}
and, with probability at least $1 - C_1 C_3 \exp(-C_2 m^{1 - \gamma \delta} y^\gamma)$, 
\begin{align}
\frac{1}{n} \sum_{k=m+1}^n X_k - C_0 k^{-\beta} \leq \frac{1}{1 - \delta} n^{-\delta} y. \label{eq:pf_Cesaro_exp_second_part_sum}
\end{align}
\end{proof}
Set $y_1 = m^{(1 - \gamma \delta)/ \gamma} y$ and $m =  n^\alpha$, with $\alpha := \gamma(1-\delta) / \{ \gamma (1 - \delta) + (1 - \gamma \delta) \}$. Observe that these choices are consistent with the conditions $y_1 \geq 1$ and $1 \leq m \leq n$. We then have that $y_1^\gamma = m^{1 - \gamma \delta} y^\gamma$ and $m^{1-\delta} / n y_1 = n^{-\delta} y$, which renders equal the right-hand sides in \eqref{eq:pf_Cesaro_exp_first_part_sum} and \eqref{eq:pf_Cesaro_exp_second_part_sum} and the corresponding exponential probability bounds. From a union bound, we then have that, with probability at least $1 - C_1 C_2 (n^\alpha + 1) \exp\{-C_2 n^{\alpha ( 1 - \gamma \delta)} y^\gamma\}$,
\begin{align}
\frac{1}{n} \sum_{k=2}^n X_k - C_0 k^{-\beta} \leq \frac{2}{1 -\delta} n^{-\delta} y.
\end{align}
We now turn to the first term of $\bar{X}_n$. We have that 
\begin{align}
P\left(\frac{1}{n} (X_1 - C_0) \geq \frac{1}{1 -\delta} n^{-\delta} y \right) 
\leq  C_1 \exp\left( - C_2 (1 - \delta)^{-\gamma} n^{1 + \gamma (1 - \delta)} y^\gamma \right).
\end{align}
Observe that $1 + \gamma (1 - \delta) > \alpha ( 1 - \gamma \delta)$. Therefore, there exists $C'_1$ that depends only on the constants of the problem ($C_0, C_1, C_2, \beta, \gamma, \delta)$, such that, for any $y \geq 1$, 
\begin{align}
&P\left(\frac{1}{n} (X_1 - C_0) \geq \frac{1}{1 -\delta} n^{-\delta} y \right) \leq C'_1 \exp( -C_2 n^{\alpha ( 1 - \gamma \delta)} y^\gamma).
\end{align}
Therefore, gathering the previous bounds via a union bound yields that there exists a constant $C_4$ that depends only on the constants of the problem such that, for any $y \geq 1$, with probability at least $1 - C_4 n^\alpha \exp\{- C_2 n^{\alpha ( 1 - \gamma \delta)} y^\gamma \}$, $\bar{X}_n \leq C_0 / (1 -  \beta) n^{-\beta} + 3 / (1-\delta) n^{-\delta} y.$

\section*{Acknowledgements}

The authors are grateful to Iosif Pinelis, who answered a question by AFB on MathOverflow \citep{pinelis_mathoverflow2019} that partly inspired our counterexample in Section~\ref{section:counterexample}. AL was supported by the National Institutes of Health (NIH) under award number DP2- LM013340. The content is solely the responsibility of the authors and does not necessarily represent the official views of the NIH.

\bibliography{biblio}

\end{document}